\documentclass[12pt]{amsart}
\usepackage[all]{xy}
\usepackage{a4wide}
\usepackage{mathtools}
\usepackage{amssymb}
\usepackage{amsthm}
\usepackage{amsmath}
\usepackage{amscd}
\usepackage{verbatim}
\usepackage[all]{xy}

\usepackage{enumitem}

\DeclareMathOperator{\reg}{reg}
\newcommand{\PSL}{\mathrm{PSL}}
\newcommand{\SL}{\mathrm{SL}}
\newcommand{\N}{\mathbb{N}}
\newcommand{\Z}{\mathbb{Z}}
\newcommand{\R}{\mathbb{R}}
\newcommand{\C}{\mathbb{C}}
\newcommand{\Q}{\mathbb{Q}}
\renewcommand{\H}{\mathbb{H}}
\newcommand{\tr}{\mathrm{tr}}
\DeclareMathOperator{\imag}{Im}
\DeclareMathOperator{\calQ}{\mathcal{Q}}

\DeclareMathOperator{\sgn}{sgn}

\renewcommand{\pmod}[1]{\  \,  \left( \mathrm{mod} \,  #1 \right)}
\newcommand{\KM}{\mathrm{KM}}
\renewcommand{\S}{\mathrm{S}}

\DeclareMathOperator{\e}{\mathfrak{e}}

\numberwithin{equation}{section}
	\newtheorem{Satz}{Satz}[section]
	\newtheorem{theorem}[Satz]{Theorem}
	\newtheorem{lemma}[Satz]{Lemma}
	\newtheorem{proposition}[Satz]{Proposition} 
	\newtheorem{corollary}[Satz]{Corollary}
	
	\theoremstyle{definition} 
	
	\newtheorem{example}[Satz]{Example}
	\newtheorem{remark}[Satz]{Remark}
\date{\today}

 \author{Claudia Alfes-Neumann}

\address{Mathematical Institute, Paderborn University, Warburger Str. 100,
D-33098 Paderborn, Germany}
\email{alfes@math.uni-paderborn.de}

\author{Markus Schwagenscheidt}

\address{Mathematical Institute, University of Cologne, Weyertal 86-90, D--50931 Cologne, Germany}

\email{mschwage@math.uni-koeln.de}
\title[]{Traces of reciprocal singular moduli}

\thanks{The research of the second author is supported by the SFB-TRR 191 \lq Symplectic Structures in Geometry, Algebra and Dynamics\rq, funded by the DFG. We thank Paloma Bengoechea, Stephan Ehlen, and Michalis Neururer for helpful discussions. Further, we thank the anonymous referee for useful remarks which helped to improve the results of this paper.}

\begin{document}

\begin{abstract}
	We show that the generating series of traces of reciprocal singular moduli is a mixed mock modular form of weight $3/2$ whose shadow is given by a linear combination of products of unary and binary theta functions. To prove these results, we extend the Kudla-Millson theta lift of Bruinier and Funke to meromorphic modular functions.
\end{abstract}

\maketitle

\section{Introduction and statement of results}
\label{sec:introduction}

The special values of the modular $j$-invariant
\[
j(z) = q^{-1} + 744 + 196884q + 21493760q^{2} + 864299970q^{3} + \dots 
\] 
at imaginary quadratic points in the upper half-plane are called singular moduli. By the theory of complex multiplication they are algebraic integers in the ray class fields of certain orders in imaginary quadratic fields. In particular, their traces
\[
\tr_{j}(D) = \sum_{Q \in \mathcal{Q}_{D}^{+}/ \Gamma}\frac{j(z_{Q})}{|\overline{\Gamma}_{Q}|}
\]
are known to be rational integers. Here $\mathcal{Q}_{D}^{+}$ denotes the set of positive definite quadratic forms of discriminant $D < 0$, on which $\Gamma = \SL_{2}(\Z)$ acts with finitely many orbits. Further, $\overline{\Gamma}_{Q}$ is the stabilizer of $Q$ in $\overline{\Gamma} = \PSL_{2}(\Z)$ and $z_{Q} \in \H$ is the CM point associated to $Q$.

In his seminal paper ``Traces of singular moduli'' (see \cite{zagiertraces}) Zagier showed that the generating series
\[
q^{-1} - 2 - \sum_{D < 0}\tr_{J}(D)q^{-D} = q^{-1}-2+248q^{3}-492q^{4}+4119q^{7}-7256q^{8}+\dots
\]
of traces of CM values of $J = j-744$ is a weakly holomorphic modular form of weight $3/2$ for $\Gamma_{0}(4)$. It follows from this result, by adding a multiple of Zagier's mock modular Eisenstein series of weight $3/2$ (see \cite{zagiereisensteinseries}), that the generating series 
\[
q^{-1} +60- \sum_{D < 0}\tr_{j}(D)q^{-D} = q^{-1}+60-864q^{4}+3375q^{7}-8000q^{8}+\dots
\]
of traces of singular moduli is a mock modular form of weight $3/2$ for $\Gamma_{0}(4)$ whose shadow is a multiple of the Jacobi theta function $\theta(\tau)=\sum_{n \in \Z}q^{n^{2}}$. These results have been generalized in various directions, for example to generating series of traces of CM values of weakly holomorphic modular functions for congruence subgroups by Bruinier and Funke \cite{bruinierfunketraces}, using the so-called Kudla-Millson theta lift. 

The starting point for the present article was the question whether the generating series of traces
\[
\tr_{1/j}(D) = \sum_{Q \in \mathcal{Q}_{D}^{+}/ \Gamma}\frac{1/j(z_{Q})}{|\overline{\Gamma}_{Q}|}
\]
of reciprocal singular moduli has similar modular transformation properties. Notice that $1/j$ has a third order pole at $\rho = e^{\pi i /3}$, so the CM value of $1/j$ at this point is not defined. However, if we replace $1/j(\rho)$ by the constant term in the elliptic expansion of $1/j$ around $\rho$ (see \eqref{eq elliptic expansion}), then $\tr_{1/j}(D)$ is defined for every $D < 0$. By the theory of complex multiplication the traces $\tr_{1/j}(D)$ are rational numbers, but they are usually not integers.

In order to obtain a convenient modularity statement we have to add a constant term, which is given by the regularized average value of $1/j$ over $\Gamma \backslash \H$,
\[
\tr_{1/j}(0) = -\frac{1}{4\pi}\int_{\Gamma \backslash \H}^{\reg}1/j(z)\frac{dx dy}{y^{2}} = -\frac{1}{2^{11}\cdot 3^{4}} =-\frac{1}{165888}.
\]
We refer to Section~\ref{section traces} for the definition of the regularized average value and to Corollary~\ref{corollary average value 1/j} for the evaluation of $\tr_{1/j}(0)$. 

To describe the shadow of the generating series of traces of reciprocal singular moduli, we require the theta functions
\begin{align*}
\theta_{7/2,h}(\tau) &= v^{-3/2}\sum_{\substack{a \in \Z \\ a \equiv h \!\!\!\pmod 3}}H_{3}\left(2\sqrt{\pi v}\frac{a}{\sqrt{3}} \right)q^{a^{2}/3}, \\
\theta_{4,h}(\tau)&= \sum_{\substack{b,c \in \Z \\ b \equiv c \!\!\!\pmod 2 \\ b \equiv h \!\!\!\pmod 3}}\left(b-i\sqrt{3}c\right)^{3}q^{b^{2}/3+c^{2}},
\end{align*}
with $\tau = u+iv \in \H$ and the Hermite polynomial $H_{3}(x) = 8x^{3}-12x$. Note that $\theta_{7/2,h}$ is (up to a non-zero constant multiple) the image under the Maass raising operator $R_{3/2} = 2i\frac{\partial}{\partial \tau}+\frac{3}{2}v^{-1}$ of the weight $3/2$ unary theta function $\theta_{3/2,h} = \sum_{a \equiv h \!\!\!\pmod 3}aq^{a^2/3}$, and $\theta_{4,h}$ is a binary theta function associated to a harmonic polynomial of degree $3$. In particular, they transform like modular forms of weight $7/2$ and of weight $4$ for $\Gamma(24)$, respectively. 

We obtain the following modularity statement.

\begin{theorem}\label{theorem main result}
	The generating series
	\[
	\sum_{D \leq 0}\tr_{1/j}(D)q^{-D} = -\frac{1}{165888} + \frac{23}{331776}q^{3} + \frac{1}{3456}q^{4} - \frac{1}{3375}q^{7} + \frac{1}{8000}q^{8} + \dots
	\]
	of traces of reciprocal singular moduli converges absolutely and locally uniformly, and defines a mixed mock modular form of weight $3/2$ and depth $2$ for $\Gamma_{0}(4)$ (in the sense of \cite{dmz,gmn}, compare Section~\ref{section mock modular forms}). Its shadow is a non-zero multiple of 
	\[
	\sum_{h \!\!\!\pmod 3} v^{7/2}\overline{\theta_{7/2,h}(\tau)}\theta_{4,h}(\tau).
	\]
\end{theorem}

In order to prove the theorem, we extend the Kudla-Millson theta lift of Bruinier and Funke \cite{bruinierfunketraces} to modular functions which are allowed to have poles in $\H$. For simplicity, we restrict our attention to those meromorphic modular functions which decay like cusp forms towards $\infty$. We let $\mathbb{S}_{0}$ be the space of all such meromorphic modular forms of weight $0$. The theta lift of $f \in \mathbb{S}_{0}$ is defined by the regularized inner product (see Section~\ref{section inner products})
\begin{equation}\label{kmliftmero}
\Phi_{\KM}(f,\tau) = \left\langle f, \overline{\Theta_{\KM}(\cdot,\tau)} \right\rangle^{\reg},
\end{equation}
where $\Theta_{\KM}(z,\tau)$ is the Kudla-Millson theta function (see Section~\ref{section theta functions}). The theta function transforms like a modular form of weight $3/2$ for $\Gamma_{0}(4)$ in $\tau$, and thus the same holds for the theta lift $\Phi_{\KM}(f,\tau)$. The technical heart of this work is the explicit computation of the Fourier expansion of $\Phi_{\KM}(f,\tau)$ (see Theorem~\ref{theorem fourier expansion}). We show that the ``holomorphic part'' of $\Phi_{\KM}(f,\tau)$ is given by the generating series of traces of CM values of $f$. On the other hand, the ``non-holomorphic part'' of $\Phi_{\KM}(f,\tau)$ can be viewed as a termwise preimage of an indefinite theta function under the lowering operator (see Corollary~\ref{corollary shadow}). In particular, we obtain an explicit formula for the image of $\Phi_{\KM}(f,\tau)$ under the lowering operator. Choosing $f = 1/j$ and doing some simplicifications (see Section~\ref{section proof main result}) then yields Theorem~\ref{theorem main result}.

\begin{remark} Theorem~\ref{theorem main result} could also be proved using results of Bringmann, Ehlen, and the second author, namely by taking the constant term in the elliptic expansion at $z = \rho$ of the function $\frac{1}{j'(z)}A^{*}(z,\tau)$ defined in Theorem 1.1 of \cite{bes}.
\end{remark}

\begin{remark}
	Similar theta lifts of meromorphic modular forms were studied by Bruinier, Imamoglu, Funke, and Li in \cite{bify} and by Bringmann and the authors of the present work in \cite{abs}. It was shown there that the generating series of traces of cycle integrals of meromorphic modular forms of positive even weight can be completed to real-analytic modular forms of half-integral weight whose images under the lowering operator are given by certain indefinite theta functions. 
\end{remark}

\begin{remark}
	In \cite{zagiertraces} Zagier also showed the weight $1/2$ modularity of twisted traces of $J= j-744$. Using the so-called Millson theta function (compare \cite{ansmillson}) one can investigate the modularity of the generating series of twisted traces of $1/j$. This is the topic of an ongoing Master's thesis under the supervision of the first author.
\end{remark}

\begin{remark} \label{remark general setup}Using a vector-valued setup as in \cite{bruinierfunketraces} one can generalize the results of the present work to arbitrary congruence subgroups. For the convenience of the reader we state the Fourier expansion of the Kudla-Millson theta lift of meromorphic modular functions for congruence subgroups in Section~\ref{section general result}.
\end{remark}

The work is organized as follows. In Section~\ref{section preliminaries} we recall the necessary facts about mixed mock modular forms, (singular) theta functions, regularized inner products of meromorphic and real-analytic modular forms, as well as traces of CM values and average values of meromorphic modular forms. In Section~\ref{section km lift} we study the Kudla-Millson theta lift $\Phi_{\KM}(f,\tau)$ of meromorphic modular forms $f \in \mathbb{S}_{0}$. We compute its Fourier expansion and determine its image under the lowering operator. In Section~\ref{section proof main result} we show how Theorem~\ref{theorem main result} can be deduced using the Kudla-Millson theta lift of $1/j$. Finally, in Section~\ref{section general result} we state the Fourier expansion of the Kudla-Millson theta lift of meromorphic modular forms for congruence subgroups.

\section{Preliminaries}
\label{section preliminaries}

\subsection{Mixed mock modular forms}\label{section mock modular forms}

We briefly recall the notion of mixed mock modular forms of higher depth from \cite{dmz}, Section~7.3, and \cite{gmn}, Section~3.2. First, a mixed mock modular form $h$ of weight $k \in \frac{1}{2}\Z$ and depth $1$ is a holomorphic function on $\H$ which is of polynomial growth at the cusps, and for which there exist finitely many holomorphic modular forms $f_j \in M_{\ell_j}$ and $g_j \in M_{2-k+\ell_j}$ for some $\ell_j \in \frac{1}{2}\Z$ such that the \emph{completion} 
\begin{align}\label{eq mock theta completion}
\widehat{h}(\tau)= h(\tau)+\sum_j f_j(\tau) g_j^{*}(\tau)
\end{align}
is a real-analytic modular form of weight $k$. Here 
\[
g^*(\tau) = (-2i)^{1-\ell}\int_{-\overline{\tau}}^{i\infty}\overline{g_j(-\overline{z})}(z+\tau)^{\ell-2}dz
\]
denotes the non-holomorphic Eichler integral of a modular form $g \in M_{\ell}$. In this case the function 
\[
\xi_k\widehat{h}(\tau) = \sum_j v^{\ell_j}\overline{f_j(\tau)} g_j(\tau)
\]
is called the \emph{shadow} of $h$. Here $\xi_k$ is the antilinear differential operator
\[
\xi_k = 2iv^k \overline{\frac{\partial}{\partial \overline{\tau}}}.
\]
A mixed mock modular form $h$ of weight $k$ and depth $n \geq 2$ is defined to be the holomorphic part of a real-analytic modular form $\widehat{h}$ of weight $k$ whose image under the $\xi$-operator is a finite linear combination of the form
\[
\xi_k \widehat{h}(\tau) = \sum_j v^{\ell_j}\overline{f_j(\tau)} g_j(\tau),
\] 
where $g_j \in M_{2-k+\ell_j}$, and $f_j$ is the modular completion of a mixed mock modular form of weight $\ell_j$ and depth $n-1$. Here, the meaning of the holomorphic part of a real-analytic modular form is deliberately kept vague in order to include many natural examples.

\begin{example}\label{example mock modular} We show that the unary theta function $\theta_{7/2,h}$ from the introduction is the completion of a mixed mock modular form of weight $7/2$ and depth $1$. To this end, we write it as
\begin{align*}
\theta_{7/2,h}(\tau) = \frac{64\pi^{3/2}}{3\sqrt{3}}\sum_{\substack{a \in \Z \\ a \equiv h \!\!\!\pmod 3}}a^3q^{a^2/3} - \frac{24\pi^{1/2}}{\sqrt{3}v}\sum_{\substack{a \in \Z \\ a \equiv h \!\!\!\pmod 3}}aq^{a^2/3}.
\end{align*}
The first sum is holomorphic on $\H$ and at the cusps, and is the mixed mock modular part of $\theta_{7/2,h}(\tau)$. The second summand is a product of $v^{-1}$ (which is the Eichler integral of a constant) and a multiple of the holomorphic unary theta series $\theta_{3/2,h}$. Furthermore, we see that the shadow of $\theta_{7/2,h}$, that is, its image under $\xi_{7/2}$, is a multiple of $
v^{3/2}\overline{\theta_{3/2,h}(\tau)}$.
\end{example}

\subsection{Quadratic forms}
\label{section quadratic forms}

We let $\mathcal{Q}_{D}$ be the set of all integral binary quadratic forms $Q = [a,b,c]$ of discriminant $D = b^{2}-4ac$, and for $D < 0$ we let $\mathcal{Q}_{D}^{+}$ be the subset of positive definite forms. The group $\Gamma = \SL_{2}(\Z)$ acts from the right on $\mathcal{Q}_{D}$ and $\mathcal{Q}_{D}^{+}$, with finitely many orbits if $D \neq 0$.

For $Q  = [a,b,c] \in \mathcal{Q}_{D}$ and $z = x+iy \in \H$ we define the quantities
\[
Q(z,1) = az^{2}+bz + c, \qquad Q_{z} = \frac{1}{y}(a|z|^{2}+bx+c).
\]
They are related by
\begin{align}\label{eq Qz differential equations}
Q_{z}^{2} = y^{-2}|Q(z,1)|^{2}-D, \qquad \frac{\partial}{\partial \overline{z}}Q_{z} = -\frac{i}{2y^{2}}Q(z,1).
\end{align}
For $D < 0$ the CM point $z_{Q} \in \H$ associated to $Q \in \mathcal{Q}_{D}$ is defined as the unique root of $Q(z,1) = 0$ in $\H$. We can factor $Q(z,1)$ as
\begin{align}\label{eq Qz1 elliptic variable}
Q(z,1) = \frac{\sqrt{|D|}}{2\imag(z_{Q})}(z-z_{Q})(z-\overline{z}_{Q}).
\end{align}

\subsection{Theta functions}
\label{section theta functions} 

The Kudla-Millson theta function is defined for $z = x+iy \in \H$ and $\tau = u+iv \in \H$ as
\begin{align}\label{eq km theta function}
\Theta_{\KM}(z,\tau)= \sum_{D\in\Z} \sum_{Q\in\calQ_D} \varphi_{\KM}(Q,z,v) e^{-2\pi i D \tau },
\end{align}
where we set
\begin{align*}
\varphi_{\KM}(Q,z,v) = \left(4vQ_z^2-\frac{1}{2\pi}\right)e^{-4\pi v \frac{|Q(z,1)|^{2}}{y^{2}}}.
\end{align*}
The function $\Theta_{\KM}(z,\tau)$ is real-analytic in both variables and transforms like a modular form of weight $0$ in $z$ for $\Gamma$ and weight $3/2$ in $\tau$ for $\Gamma_0(4)$ (see \cite{kudlamillson86}, \cite{bruinierfunkeontwogeometric}). Moreover, as a function of $z$ it decays square exponentially towards the cusp $\infty$ (see \cite{bruinierfunketraces}). 

We also require the theta function
\begin{align}\label{eq km theta function adjoint}
\Theta_{\KM}^{*}(z,\tau) = \sum_{D\in\Z} \sum_{Q\in\calQ_D} \varphi_{\KM}^{*}(Q,z,v) e^{-2\pi i D \tau },
\end{align}
where
\begin{align*}
 \varphi_{\KM}^{*}(Q,z,v) = -\frac{v^{2}}{2y^{2}}\overline{Q(z,1)}Q_{z}e^{-4\pi v \frac{|Q(z,1)|^{2}}{y^{2}}}.
\end{align*}
It is a multiple of the derivative in $z$ of the Siegel theta function $\Theta_{\S}(z,\tau)$, which is the theta function associated to $\varphi_{\S}(Q,z,v) = v\exp(-4\pi v|Q(z,1)|^{2}/y^{2})$. In particular, it transforms like a modular form of weight $2$ in $z$ for $\Gamma$ and weight $-1/2$ in $\tau$ for $\Gamma_{0}(4)$ (see \cite{borcherds}).

\subsection{Singular theta functions}
\label{section singular theta functions}
For $Q \neq 0$ we consider the function
\begin{align}\label{eq eta}
\eta_{\KM}(Q,z,v) = \frac{Q_{z}}{2\pi Q(z,1)}e^{-4\pi v \frac{|Q(z,1)|^{2}}{y^{2}}}.
\end{align}
It is the derivative of Kudla's Green function $\xi_{\KM}$ (see \cite{kudla}), which was also used in \cite{bruinierfunketraces} to compute the Fourier expansion of the Kudla-Millson theta lift of harmonic Maass forms. For $D \geq 0$ the function $\eta_{\KM}$ is real-analytic in $z$ on all of $\H$. For $D < 0$ it is only real-analytic for $z \in \H \setminus \{z_{Q}\}$, but then the difference
\begin{align*}
\eta_{\KM}(Q,z,1) - \frac{\sgn(Q_{z})\sqrt{|D|}}{2\pi Q(z,1)}
\end{align*}
extends to a real-analytic function around $z_{Q}$ which vanishes at $z_{Q}$. This can easily be proved using the first relation in \eqref{eq Qz differential equations}.

For $z \in \H$ with $Q(z,1) \neq 0$ the function $\eta_{\KM}$ is related to $\varphi_{\KM}$ and $\varphi_{\KM}^{*}$ via the lowering operator $L_{\kappa} = -2iy^{2}\frac{\partial}{\partial \overline{z}}$ by
\begin{align}\label{eq eta differential}
L_{2,z}\eta_{\KM}(Q,z,v) =  \varphi_{\KM}(Q,z,v)
\end{align}
and
\begin{align}\label{eq eta differential siegel}
L_{3/2,\tau}\eta_{\KM}(Q,z,v) = \varphi_{\KM}^{*}(Q,z,v),
\end{align}
which can be checked by a direct calculation using \eqref{eq Qz differential equations}.

The ``theta function'' formed by summing $\eta_{\KM}(Q,z,v)e^{-2\pi i D\tau}$ over $D \in \Z$ and $Q \in \mathcal{Q}_{D} \setminus \{0\}$ would behave very badly as a function of $z$ since it would have a singularity at every CM point, that is, on a dense subset of $\H$. However, for fixed $D \in \Z$ and $\varrho \in \H$ we define the function
\begin{align}\label{eq theta preimage 1}
\widetilde{\Theta}_{\KM,D}^{*}(\varrho,v) = \sum_{\substack{Q \in \mathcal{Q}_{D} \setminus\{0\}\\ z_{Q} \neq \varrho}}\eta_{\KM}(Q,\varrho,v).
\end{align}
We can imagine it as the $(-D)$-th coefficient of the ``singular theta function'' $
\widetilde{\Theta}_{\KM}^*(\varrho,\tau)$ that one would obtain by multiplying \eqref{eq theta preimage 1} with $e^{-2\pi i D\tau}$ and summing up over all $D$. However, it is in not clear (and seems to be difficult to prove) that $\widetilde{\Theta}_{\KM}^{*}(\varrho,\tau)$ converges for every $\varrho$. Hence, for simplicity, we will not work with the full function $\widetilde{\Theta}_{\KM}^{*}(\varrho,\tau)$. By a slight abuse of notation, for $m \in \N_{0}$ we also set
\begin{align}\label{eq theta preimage}
R_{2,z}^{m}\widetilde{\Theta}_{\KM,D}^{*}(\varrho,v) = \sum_{\substack{Q \in \mathcal{Q}_{D}\setminus\{0\} \\ z_{Q} \neq \varrho}}\left(R_{2,z}^{m}\eta_{\KM}(Q,z,v)\right)|_{z = \varrho},
\end{align}
where $R_{2}^{m}= R_{m} \circ R_{m-2} \circ \dots \circ R_{2}$ with $R_{2}^{0} = \text{id}$ is an iterated version of the raising operator $R_{\kappa} = 2i\frac{\partial}{\partial z}+\kappa y^{-1}$. Using \eqref{eq eta differential siegel} we obtain
\begin{align}\label{eq theta lowering}
L_{3/2,\tau}R_{2,z}^{m}\widetilde{\Theta}_{\KM,D}^{*}(\varrho,v) = R_{2,z}^{m}\Theta_{\KM,D}^{*}(\varrho,v)
\end{align}
for every $m \in \N_{0}$ and $D \in \Z$, where $R_{2,z}^{m}\Theta_{\KM,D}^{*}(\varrho,v)$ denotes the coefficient at $e^{-2\pi i D \tau}$ of $R_{2,z}^{m}\Theta_{\KM}^{*}(\varrho,\tau)$

\subsection{Regularized inner products}
\label{section inner products}
We now describe the regularized inner product in \eqref{kmliftmero}. Let $f \in \mathbb{S}_0$. We denote by $[\varrho_{1}],\dots,[\varrho_{r}] \in \Gamma \backslash \H$ the equivalence classes of the poles of $f$ on $\H$ and we choose a fundamental domain $\mathcal{F}^{*}$ for $\Gamma \backslash \H$ containing $\varrho_{1},\dots,\varrho_{r}$ such that each $\varrho_{\ell}$ lies in the interior of $\overline{\Gamma}_{\varrho_{\ell}}\mathcal{F}^{*}$. For any $\varrho \in \H$ and $\varepsilon > 0$ we let
\begin{align*}
B_{\varepsilon}(\varrho) = \left \{z \in \H: |X_{\varrho}(z)| < \varepsilon \right \}, \qquad X_{\varrho}(z)= \frac{z-\varrho}{z-\overline{\varrho}},
\end{align*}
be the $\varepsilon$-ball around $\varrho$. Let $g: \mathbb{H} \to \mathbb{C}$ be a real-analytic and $\Gamma$-invariant function, and assume that it is of moderate growth at $\infty$. We define the regularized Petersson inner product of $f$ and $g$ by
\begin{align}\label{eq inner product}
\left\langle f, g \right\rangle^{\reg}  = \lim_{\varepsilon_{1},\dots,\varepsilon_{r} \to 0}\int_{\mathcal{F}^{*}\setminus \bigcup_{\ell=1}^{r}B_{\varepsilon_{\ell}}(\varrho_{\ell})}f(z)\overline{g(z)}\frac{dxdy}{y^{2}}.
\end{align}
It was shown in Proposition 3.2 of \cite{abs} that this regularized inner product exists under the present assumptions on $f$ and $g$. In particular, the theta lift defined in \eqref{kmliftmero} converges due to the rapid decay of the Kudla-Millson theta function at $\infty$.

Recall that $f$ has an elliptic expansion of the shape
\begin{align}\label{eq elliptic expansion}
f(z) = \sum_{n \gg -\infty}c_{f,\varrho}(n)X_{\varrho}^{n}(z)
\end{align}
around every $\varrho \in \H$ (see Proposition~17 in Zagier's part of \cite{brdgza08}). In order to evaluate regularized inner products as in \eqref{eq inner product} we will typically apply Stokes' Theorem, which yields integrals over the boundaries of the balls $B_{\varepsilon_{\ell}}(\varrho_{\ell})$. To compute such boundary integrals the following formula is useful.
	
	\begin{lemma}[Lemma 4.1 in \cite{abs}]\label{lemma residue theorem}
		Let $\varrho \in \H$, let $f\in \mathbb{S}_0$ be meromorphic near $\varrho$, and let $g:\mathbb H\to\C$ be real-analytic near $\varrho$. Then we have the formula
		\[
		\lim_{\varepsilon \to 0}\int_{\partial B_{\varepsilon}(\varrho)}f(z)g(z)dz = -4\pi\sum_{n \geq 1}\frac{\imag(\varrho)^{n}}{(n-1)!}c_{f,\varrho}(-n)R_{2}^{n-1}g(\varrho),
		\]
		where $c_{f,\varrho}(-n)$ are the coefficients of the elliptic expansion \eqref{eq elliptic expansion} of $f$.
	\end{lemma}
	
\subsection{Traces of CM values and average values of meromorphic modular forms}
\label{section traces}

For $D < 0$ we define the $D$-th trace of $f \in \mathbb{S}_0$ by
\[
\tr_{f}(D) = \sum_{Q \in \mathcal{Q}_{D}^{+}/\Gamma}\frac{c_{f,z_{Q}}(0)}{|\overline{\Gamma}_{Q}|},
\]
where $c_{f,z_{Q}}(0)$ is the constant coefficient in the elliptic expansion \eqref{eq elliptic expansion} of $f$ around the CM point $z_{Q}$. If $f$ is holomorphic at $z_{Q}$ then we have $c_{f,z_{Q}}(0) = f(z_{Q})$. 

We define the trace of index $0$ of $f \in \mathbb{S}_0$ with poles at $[\varrho_{1}],\dots,[\varrho_{r}] \in \Gamma \backslash \H$ by
\[
\tr_{f}(0) = -\frac{1}{4\pi}\langle f,1 \rangle^{\reg} = -\frac{1}{4\pi}\lim_{\varepsilon_{1},\dots,\varepsilon_{r} \to 0}\int_{\mathcal{F}^{*}\setminus \bigcup_{\ell=1}^{r}B_{\varepsilon_{\ell}}(\varrho_{\ell})}f(z)\frac{dx dy}{y^{2}}.
\]
One can view $\tr_{f}(0)$ as the regularized average value of $f$ on $\Gamma \backslash \H$. Similarly as in \cite{borcherds}, Theorem~9.2, or \cite{bruinierfunketraces}, Remark~4.9, it can be evaluated in terms of special values of the real-analytic weight $2$ Eisenstein series
\[
E_{2}^{*}(z) = -\frac{3}{\pi y}+1-24\sum_{n=1}^{\infty}\sigma_{1}(n)e^{2\pi i nz}
\]
as follows.

\begin{lemma}\label{lemma average value evaluation}
	For $f \in \mathbb{S}_0$ with poles at $[\varrho_{1}],\dots,[\varrho_{r}] \in \Gamma \backslash \H$ we have
	\[
	\tr_{f}(0) = \frac{\pi}{3}\sum_{\ell=1}^{r}\frac{1}{|\overline{\Gamma}_{\varrho_{\ell}}|}\sum_{n \geq 1}\frac{\imag(\varrho_{\ell})^{n}}{(n-1)!}c_{f,\varrho_{\ell}}(-n)R_{2}^{n-1}E_{2}^{*}(\varrho_{\ell}),
	\]
	where $c_{f,\varrho}(-n)$ are the coefficients of the elliptic expansion \eqref{eq elliptic expansion} of $f$.
\end{lemma}

\begin{proof}
	The Eisenstein series satisfies $L_{2}E_{2}^{*}(z) = \frac{3}{\pi}$, so we can write
	\[
	\tr_{f}(0) = -\frac{1}{12}\lim_{\varepsilon_{1},\dots,\varepsilon_{r} \to 0}\int_{\mathcal{F}^{*}\setminus \bigcup_{\ell=1}^{r}B_{\varepsilon_{\ell}}(\varrho_{\ell})}f(z)L_{2}E_{2}^{*}(z)\frac{dx dy}{y^{2}}.
	\]
	Now using Stokes' Theorem (in the form given in Lemma~2.1 of \cite{bringmannkaneviazovska}) and Lemma~\ref{lemma residue theorem} easily gives the stated formula.
\end{proof}

Finally, for notational convenience, we set $\tr_f(D) = 0$ for $D > 0$.

\section{The Kudla-Millson theta lift}
\label{section km lift}

Let $f \in \mathbb{S}_0$ be a meromorphic modular form of weight $0$ which decays like a cusp form towards $\infty$ and let $[\varrho_{1}],\dots,[\varrho_{r}] \in \Gamma \backslash \H$ be the classes of poles of $f$ mod $\Gamma$. We define the Kudla-Millson theta lift of $f$ by 
\begin{align}\label{eq KM lift definition}
\Phi_{\KM}(f,\tau) = \left\langle f,\overline{\Theta_{\KM}(\cdot,\tau)}\right\rangle^{\reg} = \lim_{\varepsilon_{1},\dots,\varepsilon_{r} \to 0}\int_{\mathcal{F}^{*}\setminus \bigcup_{\ell=1}^{r}B_{\varepsilon_{\ell}}(\varrho_{\ell})}f(z)\Theta_{\KM}(z,\tau)\frac{dxdy}{y^{2}}.
\end{align}
Since $\Theta_{\KM}(z,\tau)$ is real-analytic in $z$ and decays square exponentially as $y$ goes to $\infty$, it follows from Proposition~3.2 of \cite{abs} that the theta lift converges for every $\tau \in \H$. In particular, it transforms like a modular form of weight $3/2$ for $\Gamma_{0}(4)$.

We now compute the Fourier expansion of the Kudla-Millson theta lift.

\begin{theorem}\label{theorem fourier expansion}
	The Fourier expansion of the Kudla-Millson theta lift of $f \in \mathbb{S}_{0}$ is given by
	\begin{align*}
	\Phi_{\KM}(f,\tau) = \sum_{D \in \Z}\left(2\tr_{f}(D) -4\pi \sum_{\ell=1}^{r}\frac{1}{|\overline{\Gamma}_{\varrho_{\ell}}|}\sum_{n\geq 1}\frac{\imag(\varrho_{\ell})^{n}}{(n-1)!}c_{f,\varrho_{\ell}}(-n)R_{2,z}^{n-1}\widetilde{\Theta}_{\KM,D}^{*}(\varrho_{\ell},v)\right)q^{-D},
	\end{align*}
	where $c_{f,\varrho}(-n)$ are the coefficients of the elliptic expansion \eqref{eq elliptic expansion} of $f$ and $R_{2,z}^{n-1}\widetilde{\Theta}_{\KM,D}^{*}(\varrho,v)$ is defined in \eqref{eq theta preimage}. Recall that we set $\tr_f(D) = 0$ for $D > 0$.
\end{theorem}

\begin{remark}
	We state the Fourier expansion of the Kudla-Millson theta lift of meromorphic modular forms for congruence subgroups in Section~\ref{section general result}.
\end{remark}

\begin{proof}[Proof of Theorem~\ref{theorem fourier expansion}] We plug in the definition of the Kudla-Millson theta function \eqref{eq km theta function} and obtain the Fourier expansion
\begin{align*}
 \Phi_{\KM}(f,\tau)= \sum_{D \in \Z}c(D,v) e^{-2\pi i D\tau}
\end{align*}
with coefficients
\begin{align}\label{eq cDv}
c(D,v) = \left(\lim_{\varepsilon_{1},\dots,\varepsilon_{r} \to 0}\int_{\mathcal{F}^{*}\setminus \bigcup_{\ell=1}^{r}B_{\varepsilon_{\ell}}(\varrho_{\ell})}f(z) \sum_{Q \in \mathcal{Q}_{D}}\varphi_{\KM}(Q,z,v)  \frac{dx dy}{y^{2}}\right).
\end{align}
We now compute the coefficients $c(D,v)$ for fixed $D \in \Z$ and $v>0$.

\subsubsection*{The coefficients of index $D > 0$} In this case the function $\eta_{\KM}(Q,z,v)$ defined in \eqref{eq eta} is real-analytic in $z$ on all of $\H$. We use the differential equation \eqref{eq eta differential} and apply Stokes' Theorem (in the form given in Lemma~2.1 of \cite{bringmannkaneviazovska}) to obtain
\begin{align}\label{eq cDv D>0}
c(D,v) = \sum_{\ell=1}^{r}\lim_{\varepsilon_{\ell} \to 0}\int_{\partial (B_{\varepsilon_{\ell}}(\varrho_{\ell}) \cap \mathcal{F}^{*})}f(z)\sum_{Q \in \mathcal{Q}_{D}}\eta_{\KM}(Q,z,v) dz.
\end{align}
Here we also used that $f$ decays like a cusp form towards $\infty$ and that all other boundary integrals cancel out in $\Gamma$-equivalent pairs due to the modularity of the integrand. Using the disjoint union 
\[
B_{\varepsilon_{\ell}}(\varrho_{\ell}) = \bigcup_{\gamma \in \overline{\Gamma}_{\varrho_{\ell}}}\gamma(B_{\varepsilon_{\ell}}(\varrho_{\ell}) \cap \mathcal{F}^{*})
\]
we see that integrating over the full boundary $\partial B_{\varepsilon_{\ell}}(\varrho_{\ell})$ on the right-hand side of \eqref{eq cDv D>0} gives an additional factor $1/|\overline{\Gamma}_{\varrho_{\ell}}|$. For $D > 0$ the function $\sum_{Q \in \mathcal{Q}_{D}}\eta_{\KM}(Q,z,v)$ is real-analytic in $z$ on $\H$, so using Lemma~\ref{lemma residue theorem} we find
\begin{align*}
c(D,v) = -4\pi\sum_{\ell=1}^{r}\frac{1}{|\overline{\Gamma}_{\varrho_{\ell}}|}\sum_{n \geq 1}\frac{\imag(\varrho_{\ell})^{n}}{(n-1)!}c_{f,\varrho_{\ell}}(-n)R_{2,z}^{n-1}\left(\sum_{Q \in \mathcal{Q}_{D}}\eta_{\KM}(Q,z,v) \right)\bigg|_{z = \varrho_{\ell}}.
\end{align*}
This finishes the computation in the case $D > 0$.

\subsubsection*{The coefficients of index $D = 0$} We split off the summand for $Q = 0$ in \eqref{eq cDv}, which yields
\[
-\frac{1}{2\pi}\lim_{\varepsilon_{1},\dots,\varepsilon_{r} \to 0}\int_{\mathcal{F}^{*}\setminus \bigcup_{\ell=1}^{r}B_{\varepsilon_{\ell}}(\varrho_{\ell})}f(z)\frac{dx dy}{y^{2}} 
= 2\tr_{f}(0).
\]
The remaining part with $Q \neq 0$ can be computed as in the case $D > 0$.

\subsubsection*{The coefficients of index $D < 0$}

Let us first suppose that $f$ does not have a pole at any CM point of discriminant $D$. Let $[z_{1}],\dots,[z_{s}] \in \Gamma \backslash \H$ be the $\Gamma$-classes of CM points of discriminant $D$, and let $Q_{1},\dots,Q_{s} \in \mathcal{Q}_{D}^{+}$ be the corresponding quadratic forms. We can assume that $z_{1},\dots,z_{s} \in \mathcal{F}^{*}$ and that every $z_{\ell}$ lies in the interior of $\overline{\Gamma}_{z_{\ell}}\mathcal{F}^{*}$. We cut out a ball $B_{\delta_{\ell}}(z_{\ell})$ around every CM point $z_{\ell}$ and then apply Stokes' Theorem as in the case $D > 0$ to obtain
\begin{align}\label{eq cDv D<0}
\begin{split}
c(D,v) &= \sum_{\ell=1}^{r}\frac{1}{|\overline{\Gamma}_{\varrho_{\ell}}|}\lim_{\varepsilon_{\ell} \to 0}\int_{\partial B_{\varepsilon_{\ell}}(\varrho_{\ell})}f(z)\sum_{Q \in \mathcal{Q}_{D}}\eta_{\KM}(Q,z,v) dz \\
&\quad + \sum_{\ell = 1}^{s}\frac{1}{|\overline{\Gamma}_{z_{\ell}}|}\lim_{\delta_{\ell} \to 0}\int_{\partial B_{\delta_{\ell}}(z_{\ell})}f(z)\sum_{Q \in \mathcal{Q}_{D}}\eta_{\KM}(Q,z,v) dz.
\end{split}
\end{align}
The first line can be evaluated as in the case $D > 0$. In the second line, for fixed $\ell$ the summands with $Q \neq \pm Q_{\ell}$ are real-analytic near $z_{\ell}$, so their integrals vanish as $\delta_{\ell} \to 0$. Replacing $-Q_{\ell}$ by $Q_{\ell}$ gives a factor $2$. Further, since $\eta_{\KM}(Q,z,v)-\sgn(Q_{z})\sqrt{|D|}/2\pi Q(z,1)$ is real analytic near the CM point $z_{Q}$, the second line in \eqref{eq cDv D<0} is given by
\[
2\sum_{\ell = 1}^{s}\frac{1}{|\overline{\Gamma}_{z_{\ell}}|}\lim_{\delta_{\ell} \to 0}\int_{\partial B_{\delta_{\ell}}(z_{\ell})}f(z)\frac{\sgn((Q_{\ell})_{z})\sqrt{|D|}}{2\pi Q_{\ell}(z,1)}dz.
\]
Note that $\sgn((Q_{\ell})_{z}) = 1$ for $z$ close to $z_{\ell}$. If we now plug in the elliptic expansion \eqref{eq elliptic expansion} of $f$ and the expression \eqref{eq Qz1 elliptic variable} for $Q(z,1)$, we arrive at
\begin{align}\label{eq tr}
2\sum_{\ell = 1}^{s}\frac{1}{|\overline{\Gamma}_{z_{\ell}}|}\sum_{n \geq 0}\frac{\imag(z_{\ell})}{\pi}c_{f,z_{\ell}}(n)\lim_{\delta_{\ell} \to 0}\int_{\partial B_{\delta_{\ell}}(z_{\ell})}\frac{(z-z_{\ell})^{n-1}}{(z-\overline{z}_{\ell})^{n+1}}dz.
\end{align}
By the residue theorem, the last integral vanishes unless $n = 0$, in which case it equals $\pi/\imag(z_{\ell})$. Hence the second line in \eqref{eq cDv D<0} is given by 
\[
2\sum_{\ell = 1}^{s}\frac{c_{f,z_{\ell}}(0)}{|\overline{\Gamma}_{z_{\ell}}|} = 2\tr_{f}(D),
\]
which finishes the computation for $D < 0$ if $f$ does not have a pole at any CM point of discriminant $D$.

We now indicate the changes in the computation if $f$ has poles at some CM points of discriminant $D$. For notational convenience, let us assume that $\varrho_{1} = z_{1}$ is a pole of $f$ which is also a CM point. In this case we do not need to cut out an additional $\delta_{1}$-ball around $z_{1}$ before applying Stokes' Theorem since we already cut out an $\varepsilon_{1}$-ball around $\varrho_{1}$. In particular, the summand for $z_{1}$ in the second line of \eqref{eq cDv D<0} has to be omitted, and the summand for $\varrho_{1}$ in the first line of \eqref{eq cDv D<0} has to be computed as follows. We write
\begin{align}\label{eq with poles}
\begin{split}
&\lim_{\varepsilon_{1} \to 0}\int_{\partial B_{\varepsilon_{1}}(\varrho_{1})}f(z)\sum_{Q \in \mathcal{Q}_{D}}\eta_{\KM}(Q,z,v) dz \\
&\quad= 2\lim_{\varepsilon_{1} \to 0}\int_{\partial B_{\varepsilon_{1}}(\varrho_{1})}f(z)\frac{\sqrt{|D|}}{2\pi Q_{1}(z,1)} dz \\
&\qquad + \lim_{\varepsilon_{1} \to 0}\int_{\partial B_{\varepsilon_{1}}(\varrho_{1})}f(z)\sum_{Q \in \mathcal{Q}_{D}}\left(\eta_{\KM}(Q,z,v)-\delta_{\varrho_{1} = z_{Q}}\frac{\sgn(Q_{z})\sqrt{|D|}}{2\pi Q(z,1)}\right) dz.
\end{split}
\end{align}
The factor $2$ in the second line comes from the fact that $Q_{1}$ and $-Q_{1}$ have the same CM point $z_{1}$, and the sign is gone since $\sgn((Q_{1})_{z}) = 1$ for $z$ close to $z_{1}$. 

The sum over $Q \in \mathcal{Q}_{D}$ in the third line of \eqref{eq with poles} is real-analytic near $\varrho_{1}$, hence the expression in the third line can be computed as in the case $D > 0$ to
\begin{align*}
-4\pi\sum_{n \geq 1}\frac{\imag(\varrho_{1})^{n}}{(n-1)!}c_{f,\varrho_{1}}(-n)R_{2,z}^{n-1}\left(\sum_{Q \in \mathcal{Q}_{D}}\left(\eta_{\KM}(Q,z,v)-\delta_{\varrho_{1} = z_{Q}}\frac{\sgn(Q_{z})\sqrt{|D|}}{2\pi Q(z,1)} \right) \right)\bigg|_{z = \varrho_{1}}.
\end{align*}
For $\varrho_{1} = z_{Q}$ we have
\[
R_{2,z}^{n-1}\left(\eta_{\KM}(Q,z,v)-\delta_{\varrho_{1} = z_{Q}}\frac{\sgn(Q_{z})\sqrt{|D|}}{2\pi Q(z,1)} \right)\bigg|_{z = \varrho_{1}} = 0
\]
for all $n \geq 1$, which follows from the fact that the difference in the brackets can be written as a real-analytic multiple of $\overline{Q(z,1)}$. Therefore we can just omit the summands for $\varrho_{1} = z_{Q}$.

In the second line in \eqref{eq with poles}, we plug in the elliptic expansion of $f$ and obtain
\begin{align*}
2\lim_{\varepsilon_{1} \to 0}\int_{\partial B_{\varepsilon_{1}}(\varrho_{1})}f(z)\frac{\sqrt{|D|}}{2\pi Q_{1}(z,1)} dz = 2\sum_{n \gg -\infty}\frac{\imag(\varrho_{1})}{\pi}c_{f,\varrho_{1}}(n)\lim_{\varepsilon_{1} \to 0}\int_{\partial B_{\varepsilon_{1}}(\varrho_{1})}\frac{(z-\varrho_{1})^{n-1}}{(z-\overline{\varrho}_{1})^{n+1}}dz.
\end{align*}
Note that, in comparison to the expression in \eqref{eq tr}, the sum now has finitely many terms with negative $n$. However, using the residue theorem one can check that it is still true that the last integral vanishes unless $n = 0$, in which case it equals $\pi/\imag(\varrho_{1})$. Hence the last displayed formula becomes $2c_{f,\varrho_{1}}(0)$, which contributes to the trace of index $D$. 

We can proceed in the same way for every pole of $f$ which is also a CM point of discriminant $D$. The proof is finished.
\end{proof}

From the Fourier expansion of the Kudla-Millson theta lift and \eqref{eq theta lowering} we immediately obtain its image under the lowering operator.

\begin{corollary}\label{corollary shadow}
	The image under the lowering operator $L_{3/2,\tau}$ of the Kudla-Millson theta lift of $f \in \mathbb{S}_{0}$ is given by
	\[
	L_{3/2,\tau}\Phi_{\KM}(f,\tau) = -4\pi\sum_{\ell = 1}^{r}\frac{1}{|\overline{\Gamma}_{\varrho_{\ell}}|}\sum_{n\geq 1}\frac{\imag(\varrho_{\ell})^{n}}{(n-1)!}c_{f,\varrho_{\ell}}(-n)R_{2,z}^{n-1}\Theta_{\KM}^{*}(\varrho_{\ell},\tau).
	\]
\end{corollary}

\section{The proof of Theorem \ref{theorem main result}}
\label{section proof main result}

We now prove Theorem~\ref{theorem main result}. Let $\rho = e^{\pi i /3}$. By Theorem~\ref{theorem fourier expansion} the Fourier expansion of the Kudla-Millson theta lift of $1/j$ is given by 
\begin{align*}
	\Phi_{\KM}(1/j,\tau) = \sum_{D \in \Z}\left(2\tr_{1/j}(D) -4\pi \frac{1}{|\overline{\Gamma}_{\rho}|}\sum_{n\geq 1}\frac{\imag(\rho)^{n}}{(n-1)!}c_{1/j,\rho}(-n)R_{2,z}^{n-1}\widetilde{\Theta}_{\KM,D}^{*}(\rho,v)\right)q^{-D}.
	\end{align*}
	
	We first consider the growth of the traces of reciprocal singular moduli.
	
	\begin{proposition}\label{proposition polynomial growth}
	The traces $\tr_{1/j}(D)$ are of polynomial growth in $|D|$. In particular, the generating series of traces of reciprocal singular moduli converges absolutely and locally uniformly.
\end{proposition}

\begin{proof}
	We start with the simple estimate
	\[
	|\tr_{1/j}(D)| \leq H(D)\max_{Q \in \mathcal{Q}_{D}^{+}/\Gamma}|c_{1/j,z_Q}(0)|,
	\]
	where $H(D) = \sum_{Q \in \mathcal{Q}_{D}^{+}/\Gamma}1/|\overline{\Gamma}_{Q}|$ is the $D$-th Hurwitz class number. It follows from Dirichlet's class number formula that $H(D)$ is of polynomial growth in $|D|$. Since the value $|c_{1/j,\rho}(0)|$ does not grow with $|D|$, it remains to estimate $|c_{1/j,z_Q}(0)| = |1/j(z_Q)|$ for CM points $z_Q \in \mathcal{Q}_{D}^{+}$ with $z_Q \neq \rho$ in terms of $|D|$. We can assume that all the CM points $z_Q$ lie inside the rectangle $\{z \in \H: 0 \leq x \leq 1, y \geq 1/2\}$, which contains $\rho$ in its interior but does not contain any of the other $\Gamma$-translates of $\rho$. If we cut out a small $\varepsilon$-ball around $\rho$ from this rectangle, then $1/j$ is bounded on the remaining set since it has no other poles there and decays like a cusp form towards $\infty$. In particular, it suffices to estimate $|1/j(z_Q)|$ on the $\varepsilon$-ball around $\rho$. If we look at the Taylor expansion $1/j(z) = \sum_{n \gg -\infty}\alpha_{1/j,\varrho}(n)(\rho-z)^n$ of $1/j$ on this $\varepsilon$-ball, we see that it suffices to estimate $|\rho-z_Q|^{-1}$ in terms of $|D|$. We can write
	\[
	\rho = \frac{1}{2}+i\frac{\sqrt{3}}{2}, \qquad z_Q = -\frac{b}{2a}+i\frac{\sqrt{|D|}}{2a},
	\]
	with $a,b,c \in \Z$. Then we have
	\[
	\frac{1}{|\rho-z_Q|^2} =  \frac{4a^2}{(a+b)^2 + (\sqrt{3}a-\sqrt{|D|})^2}.
	\]
	Now there are two cases. If $(a+b)^2 \neq 0$, then it is a least $1$, and the whole denominator is at least $1$. Furthermore, the assumption that $\imag(z_Q) \geq 1/2$ implies $a \leq \sqrt{|D|}$. Hence, in this case we get
	\[
	\frac{1}{|\rho-z_Q|^2} \leq 4a^2 \leq 4|D|.
	\]
	If, on the other hand, we have $(a+b)^2 = 0$, then
	\[
	\frac{1}{|\rho-z_Q|^2} = \frac{4a^2}{(\sqrt{3}a-\sqrt{|D|})^2} = \frac{4a^2(\sqrt{3}a+\sqrt{|D|})^2}{(3a^2-|D|)^2}.
	\]
	Since the denominator is not $0$ (as we assumed $z_Q \neq \rho$), it is at least $1$, so we find
	\[
	\frac{1}{|\rho-z_Q|^2} \leq 4a^2(\sqrt{3}a+\sqrt{|D|})^2 \leq 4(\sqrt{3}+1)^2|D|^2.
	\]
	In any case, we see that $|\rho-z_Q|^{-1}$ is bounded by a polynomial in $|D|$, which finishes the proof.
\end{proof}
%

We now compute the coefficients $c_{1/j,\rho}(-n)$ for $n\geq 1$ in the elliptic expansion of $1/j$ at $\rho$. They can be explicitly described in terms of the Chowla-Selberg period
\begin{align}\label{eq Chowla Selberg}
	\Omega = \Omega_{\Q(\sqrt{-3})} =  \frac{1}{\sqrt{6\pi}}\left(\frac{\Gamma(\frac{1}{3})}{\Gamma(\frac{2}{3})} \right)^{\frac{3}{2}} \approx 0.6409273802
	\end{align}
of $\Q(\sqrt{-3})$ (compare Section~6.3 in Zagier's part of \cite{brdgza08}).

\begin{proposition}\label{proposition 1/j expansion}
	The function $1/j$ has an elliptic expansion at $\rho$ of the form
	\[
	-\frac{\pi^{-3}\Omega^{-6}}{2^{12}\cdot 3^{3}}X_{\rho}^{-3}(z) + \frac{23}{2^{12}\cdot 3^{3}} + O(X_{\rho}(z)).
	\]
\end{proposition}

\begin{proof}
	The $j$-function has an elliptic expansion at $\rho$ of the form
	\[
	j(z) = \sum_{n=3}^{\infty}R_{0}^{n}j(\rho)\frac{\imag(\rho)^{n}}{n!}X_{\rho}^{n}(z),
	\]
	compare Proposition~17 in Zagier's part of \cite{brdgza08}. Using the theory of complex multiplication one can show that the values $R_{0}^{n}j(\rho)$ are algebraic multiples of $\pi^{n}\Omega^{2n}$ (compare Proposition~26 in Zagier's part of \cite{brdgza08}). Explicitly, we have
	\begin{align*}
	R_{0}^{3}j(\rho) &= -2^{16}\cdot 3^{2}\cdot\sqrt{3}\cdot \pi^{3}\Omega^{6},\\
	R_{0}^{4}j(\rho) &= R_{0}^{5}j(\rho) = 0,  \\
	R_{0}^{6}j(\rho) &= -2^{22} \cdot 3^2 \cdot 5 \cdot 23 \cdot \pi^{6}\Omega^{12}.
	\end{align*}
	Note that $R_{0}^{4}j(\rho) = R_{0}^{5}j(\rho) = 0$ follows from the simple fact that every function which transforms like a modular form of weight $k \not\equiv 0 \pmod 6$ for $\Gamma$ vanishes at $\rho$. The other values can be computed, for example, by writing the almost holomorphic modular forms $R_{0}^{3}j(z)\Delta(z)$ and $R_{0}^{6}j(z)\Delta(z)$ in terms of the Eisenstein series $E_{2}^{*},E_{4}$ and $E_{6}$ and using their values at $\rho$ given in the table after Proposition 27 in Zagier's part of \cite{brdgza08}. We also checked the above evaluations numerically.
	
	By inverting the elliptic expansion of $j$ we obtain the elliptic expansion of $1/j$ at $\rho$.
\end{proof}



\begin{corollary}\label{corollary trace -3}
	We have 
	\[
	\tr_{1/j}(-3) = \frac{23}{2^{12}\cdot 3^{4}} = \frac{23}{331776}.
	\]
\end{corollary}

\begin{proof}
	By definition $\tr_{1/j}(-3) = \frac{c_{1/j,\rho}(0)}{3}$, so the result follows from Proposition~\ref{proposition 1/j expansion}.
\end{proof}


\begin{corollary}\label{corollary average value 1/j}
	The trace of index $0$ of $1/j$ is given by
	\[
	\tr_{1/j}(0) = - \frac{1}{2^{11}\cdot 3^{4}}=  -\frac{1}{165888}.
	\]
\end{corollary}

\begin{proof}
We have the special values
\[
E_{2}^{*}(\rho) = R_{2}E_{2}^{*}(\rho) = 0, \quad R_{2}^{2}E_{2}^{*}(\rho) = \frac{32}{\sqrt{3}}\pi^{2}\Omega^{6}.
\]
They can be obtained, for example, by writing $E_{2}^{*}, R_{2}E_{2}^{*}$ and $R_{2}^{2}E_{2}^{*}$ in terms of the Eisenstein series $E_{2}^{*},E_{4}$ and $E_{6}$ (using Proposition 15 in Zagier's part of \cite{brdgza08}) and then plugging in their values at $\rho$ from the table after Proposition 27 in loc. cit.
 Combining the above values with Lemma~\ref{lemma average value evaluation} and Proposition~\ref{proposition 1/j expansion} we obtain the value of $\tr_{1/j}(0)$.
\end{proof}

Next, we simplify the shadow of $\Phi_{\KM}(1/j,\tau)$. By Corollary~\ref{corollary shadow} and the elliptic expansion of $1/j$ from Proposition~\ref{proposition 1/j expansion} the image of $\Phi_{\KM}(1/j,\tau)$ under the lowering operator is given by
\[
-\frac{2\pi}{|\overline{\Gamma}_{\rho}|}\imag(\rho)^{3}c_{1/j,\rho}(-3)R_{2,z}^{2}\Theta_{\KM}^{*}(\rho,\tau).
\]
Computing the action of the raising operators we get
\begin{align*}
R_{2,z}^{2}\Theta_{\KM}^{*}(\rho,\tau) = \sum_{D \in \Z}\sum_{Q \in \mathcal{Q}_{D}}\frac{\overline{Q(\rho,1)}^{3}}{\imag(\rho)^{6}}\left(12\pi v^{3}Q_{\rho}-32\pi^{2}v^{4}Q_{\rho}^{3} \right)e^{-4\pi v \frac{|Q(\rho,1)|^{2}}{\imag(\rho)^{2}}}e^{-2\pi i D \tau}.
\end{align*}
If we write $Q = [A,B,C]$ and $\rho = \frac{1}{2}+i\frac{\sqrt{3}}{2}$ then we have the evaluations
\begin{align*}
Q_{\rho} = \frac{1}{\sqrt{3}}\left(2A+B+2C \right), \qquad \overline{Q(\rho,1)} = \frac{1}{2}(-A + B + 2C)-i\frac{\sqrt{3}}{2}(A+B).
\end{align*}
We set
\[
a = 2A+B+2C, \qquad b = -A + B + 2C, \qquad c = A+B.
\]
It is then not hard to see that with $A,B,C \in \Z$ the pairs $(a,b,c)$ run through the sublattice of $\Z^{3}$ defined by the conditions $a \equiv b \!\!\!\pmod 3$ and $b\equiv c \!\!\!\pmod 2$.
Thus we obtain
\begin{align*}
R_{2,z}^{2}\Theta_{\KM}^{*}(\rho,\tau) = \frac{8}{81\sqrt{3}}\sum_{\substack{a,b,c \in \Z \\ a \equiv b \!\!\!\pmod 3 \\ b\equiv c \!\!\! \pmod 2}}\left(36\pi v^{3}a-32\pi^{2}v^{4}a^{3}\right)\left(b-i\sqrt{3}c \right)^{3}e^{2\pi i a^{2}\tau/3}e^{-2\pi i (b^{2}/3+c^{2})\overline{\tau}}.
\end{align*}
Splitting the sum over $a$ into arithmetic progressions mod 3 and using the relation $\xi_{3/2} = v^{-1/2}\overline{L_{3/2}}$ we easily obtain the shadow given in Theorem~\ref{theorem main result}. 

Although we did not rigorously define what the holomorphic part of a real-analytic modular form should be, it seems reasonable from the Fourier expansion of $\Phi_{\KM}(1/j,\tau)$ to view the generating series of traces of reciprocal singular moduli as its holomorphic part. Since the image of $\Phi_{\KM}(1/j,\tau)$ under the $\xi$-operator is a linear combination of products of the functions $v^{7/2}\overline{\theta_{7/2,h}(\tau)}\theta_{4,h}(\tau)$, and $\theta_{7/2,h}$ is the completion of a mixed mock modular form of weight $7/2$ and depth $1$, see Example~\ref{example mock modular}, we may view the generating series of traces of reciprocal singular moduli as a mixed mock modular form of weight $3/2$ and depth $2$ by the definition given in Section~\ref{section mock modular forms}.
This finishes the proof of Theorem~\ref{theorem main result}.

\begin{remark}
	The above arguments work more generally for meromorphic modular forms $f \in \mathbb{S}_0$ having poles only at CM points in $\H$. In particular, one can generalize Theorem~\ref{theorem main result} to such $f$, that is, the generating series
	\[
	\sum_{D \leq 0}\tr_f(D)q^{-D}
	\]
	of traces of CM values of $f$ converges absolutely and locally uniformly, and defines a mixed mock modular form of weight $3/2$ and higher depth whose shadow is a linear combination of products of unary and binary theta functions. Indeed, the proof of Proposition~\ref{proposition polynomial growth} works in the same way if we replace $\rho$ by any other CM point, which shows that $\tr_f(D)$ is of polynomial growth in $|D|$.
	Furthermore, the above splitting of $R_{2,z}^2 \Theta_{\KM}^{*}(\rho,\tau)$ into a sum of products of unary and binary theta functions is a special case of a more general principle. In particular, a similar splitting exists for $R_{2,z}^{n-1} \Theta_{\KM}^{*}(z_0,\tau)$ for any CM point $z_0 \in \H$ and $n \geq 1$. First, one can show by induction that $R_{2,z}^{n-1}\Theta_{\KM}^{*}(z,\tau)$ is a non-zero constant multiple of
	\[
	v^{\frac{n}{2}+1}\sum_{D \in \Z}\sum_{Q \in \mathcal{Q}_D}\frac{\overline{Q(z,1)}^n}{y^{2n}}H_{n}\left(2\sqrt{\pi v}Q_z\right)e^{-4\pi v \frac{|Q(z,1)|^2}{y^2}}e^{-2\pi i D \tau},
	\] 
	where $H_n(x)$ denotes the $n$-th Hermite polynomial. If we plug in a CM point $z_0 \in\H$, the function $R_{2,z}^{n-1}\Theta_{\KM}^*(z_0,\tau)$ will split into a linear combination of products of complex conjugates of unary theta functions of weight $n+1/2$ (associated to the Hermite polynomial $H_n(x)$) and binary theta functions of weight $n+1$ (associated to harmonic homogeneous polynomials of degree $n$).
	
	 Finally, we remark that the splitting of $R_{2,z}^{n-1} \Theta_{\KM}^{*}(z_0,\tau)$ at a general CM point $z_0 \in \H$ can be described most conveniently in the vector-valued setup alluded to in Section~5 below. For an instance of the general principle, we refer the reader to \cite{ehlenduke}, where an analogous splitting of the Siegel theta function at an arbitrary CM point is worked out in the vector-valued setting.
\end{remark}

\section{The Fourier expansion of the Kudla-Millson theta lift of meromorphic modular forms for congruence subgroups}\label{section general result}

	As mentioned in Remark~\ref{remark general setup}, the results of this work can easily be generalized to arbitrary congruence subgroups by using a vector-valued setup as in \cite{bruinierfunketraces}. For the convenience of the reader we state the Fourier expansion of the Kudla-Millson theta lift of a meromorphic modular form in the general case. However, the necessary computations are analogous to the ones given in the proof of Theorem~\ref{theorem fourier expansion}, so we leave the details to the reader.
	
	Let $V$ be the real quadratic space of signature $(1,2)$ given by the set
	\[
	V = \left\{Q = \begin{pmatrix}b/2 & c \\ -a & -b/2\end{pmatrix}: a,b,c \in \R \right\},
	\]
	equipped with the quadratic form $q(Q) = \det(Q)$ and the associated bilinear form $(Q_1,Q_2) = -\tr(Q_1Q_2)$. We can also think of the elements of $V$ as binary quadratic forms $ax^2+bxy+cy^2$, with $q$ being $-\frac{1}{4}$ times the discriminant. The group $\SL_2(\R)$ acts on $V$ as isometries by conjugation, and this action is compatible with the usual action of $\SL_2(\R)$ on binary quadratic forms.
	
	Let $L \subset V$ be an even lattice with dual lattice $L'$, and let $\C[L'/L]$ be its group ring with standard basis $(\e_{h})_{h \in L'/L}$. We let $\Lambda$ be a congruence subgroup of $\SL_2(\Z)$ which acts on $L$ and fixes the classes of $L'/L$. For simplicity we assume that $-1 \in \Lambda$. For $h \in L'/L$ and $m \in \Q$ we let $L_{m,h}$ be the set of all $X \in L+h$ with $q(X) = m$, on which $\Lambda$ acts with finitely many orbits if $m \neq 0$. For $z = x+iy \in \H$ and $Q = \left(\begin{smallmatrix}b/2 &c \\ -a & -b/2\end{smallmatrix} \right)\in V$ we define the quantities
	\[
	Q(z) = az^2+bz+c, \qquad Q_z = \frac{1}{y}(a|z|^2+bx+c).
	\]
	For $Q \in L_{m,h}$ with $m > 0$ we let $z_Q \in \H$ denote the unique root of $Q(z)$ in $\H$. Given $h \in L'/L, m \in \Q_{>0}$, and a meromorphic modular form $f \in\mathbb{S}_0(\Lambda)$ of weight $0$ for $\Lambda$ which decays like a cusp form at all cusps, we define the trace function
	\[
	\tr_{f}(m,h) = \frac{1}{2}\sum_{Q \in \Lambda \backslash L_{m,h}}\frac{f(z_Q)}{|\overline{\Lambda}_{Q}|},
	\]
	where $f(z_Q)$ is again defined as the constant term in the elliptic expansion of $f(z)$ at $z = z_Q$ if $z_Q$ is a pole of $f$. Further, for $m = 0$ and $h\neq 0$ we set $\tr_f(m,h) = 0$, and we define $\tr_f(0,0)$ as the regularized average value of $f$ analogously as in Section~\ref{section traces}, but with the fundamental domain $\mathcal{F}^{*}$ for $\Gamma = \SL_2(\Z)$ replaced with a suitable fundamental domain $\mathcal{F}^{*}(\Lambda)$ for $\Lambda$. Finally, we set $\tr_f(m,h) = 0$ for $m < 0$. 
	
	The $\C[L'/L]$-valued Kudla-Millson theta function is defined by
	\[
	\Theta_{\KM}(z,\tau) = \sum_{h \in L'/L}\sum_{Q \in L+h}\left(vQ_z^2-\frac{1}{2\pi}\right)e^{-\pi v\frac{|Q(z)|^2}{y^2}}e^{2\pi i q(Q)\tau}\e_h, 
	\]
	and the Kudla-Millson theta lift of a meromorphic modular form $f \in \mathbb{S}_{0}(\Lambda)$ of weight $0$ for $\Lambda$ is defined analogously as in \eqref{eq KM lift definition}, but with the fundamental domain $\mathcal{F}^*$ replaced by $\mathcal{F}^*(\Lambda)$. The Kudla-Millson theta lift transforms like a modular form of weight $3/2$ for the Weil representation $\rho_L$. Furthermore, for fixed $\varrho \in \H$ and $h \in L'/L,m \in \Q,$ we define
	\[
	\widetilde{\Theta}_{\KM,m,h}^*(\varrho,\tau) = \sum_{\substack{Q \in L_{m,h}\setminus\{0\} \\ z_Q \neq \varrho}}\frac{Q_\varrho}{2\pi Q(\varrho)}e^{-\pi v \frac{|Q(\varrho)|^2}{\imag(\varrho)^2}},
	\]
	and we define its raised version similarly as in \eqref{eq theta preimage}.
	
	In this setting, the Fourier expansion of the Kudla-Millson theta lift is given follows.
	
	\begin{theorem}\label{theorem fourier expansion general}
	The Fourier expansion of the Kudla-Millson theta lift of $f \in \mathbb{S}_0(\Lambda)$ is given by
	\begin{align*}
	\Phi_{\KM}(f,\tau) &= \sum_{h \in L'/L}\sum_{m \in \Q}\bigg(2\tr_{f}(m,h) \\
	&\qquad\qquad-4\pi \sum_{\ell=1}^{r}\frac{1}{|\overline{\Lambda}_{\varrho_{\ell}}|}\sum_{n\geq 1}\frac{\imag(\varrho_{\ell})^{n}}{(n-1)!}c_{f,\varrho_{\ell}}(-n)R_{2,z}^{n-1}\widetilde{\Theta}_{\KM,m,h}^{*}(\varrho_{\ell},v)\bigg)q^{m}\e_h,
	\end{align*}
	where $c_{f,\varrho_\ell}(-n)$ are the coefficients of the elliptic expansions of $f$ around its poles $\varrho_1,\dots,\varrho_r$ mod $\Lambda$.
	\end{theorem}
	
	Finally, consider the lattice $L$ given by the set of all integral traceless $2$ by $2$ matrices, and choose $\Lambda = \SL_2(\Z)$. Then $L_{m,h}$ can be identified with the set $\mathcal{Q}_{-4m}$ of all integral binary quadratic forms of discriminant $-4m$. For $m > 0$ the point $z_Q \in \H$ associated to $Q\in L_{m,h}$ is precisely the CM point associated to the binary quadratic form corresponding to $Q$. The discriminant group of $L$ is isomorphic to $\Z/2\Z$. By the results of \cite{ez}, Section 5, we can identify vector-valued modular forms for the Weil representation $\rho_L$ with scalar-valued modular forms satisfying the Kohnen plus space condition via the map $f_0(\tau) \e_0 + f_1(\tau) \e_1 \mapsto f_0(4\tau)+f_1(4\tau)$. In this way Theorem~\ref{theorem fourier expansion general} generalizes Theorem~\ref{theorem fourier expansion}.

\bibliographystyle{alpha}
\bibliography{references.bib}

\end{document}